\documentclass[12pt]{amsart}
\usepackage[T1]{fontenc}
\usepackage{amsmath}
\usepackage{amscd}
\usepackage{amssymb}
\let\SavedRightarrow=\Rightarrow
\usepackage{marvosym}
\let\Rightarrow=\SavedRightarrow

\newcommand{\cov}{\operatorname{cov}}

\newcommand{\Aaa }{\mathcal A}
\newcommand{\Raa }{\mathcal R}
\newcommand{\Bee }{\mathcal B}

\newcommand{\Oee }{\mathcal O}
\newcommand{\Pee }{\mathcal P}
\newcommand{\Fee }{\mathcal F}
\newcommand{\nat }{\mathbb N}

\newcommand{\Tee }{\mathcal T}

\newcommand{\cl}{\operatorname{cl}}
\renewcommand{\int}{\operatorname{Int}}

\newtheorem{theorem}{Theorem}
\newtheorem{corollary}[theorem]{Corollary}

\newtheorem{remark}{Remark}

\author{Piotr Kalemba} 
\address{ Piotr Kalemba\\  Institute of Mathematics, University of Silesia \\  Bankowa 14, 40-007 Katowice} \email{piotr.kalemba@us.edu.pl} 
\author{Andrzej Kucharski} 
\address{Andrzej Kucharski \\  Institute of Mathematics, University of Silesia \\  Bankowa 14, 40-007 Katowice} \email{akuchar@math.us.edu.pl}

\begin{document}

\title{ Universally Kuratowski-Ulam spaces and  open-open games} 
\thanks{Andrzej Kucharski thanks European Science
Foundation for their support through the grant
3007 within the INFTY program.}
\begin{abstract}
We examine the class of  spaces in which the second player has a winning strategy in the open-open game. We show that this spaces are not  universally Kuratowski-Ulam. We also show that the games $G$ and $G_7$ introduced by  P. Daniels, K. Kunen, H. Zhou \cite{dkz} are not equivalent.
\end{abstract}

\maketitle

\section{Introduction}

First we shall recall some game introduced in \cite{dkz} called $G_2$.  Let  $X$ be a topological space equipped with a topology $\Tee$ and let $\Bee\subseteq \Tee$ be its base. The length of the game is $\omega$. 
Two players I and II take turns playing. At the n-th move II chooses a family $\Pee_n$ 
consisting of open non-empty subset of $X$ such that $\cl\bigcup\Pee_n=X$, then I picks a $V_n\in\Pee_n$.
I wins iff $\cl\bigcup_{n\in\omega} V_n=X$. Otherwise player II  wins. 
Denote by $D_{cov}$ a collection of families $\mathcal F$ consisting  of open sets with $\cl\bigcup\mathcal F= X$. We say that $\sigma_{cov}:(\bigcup D_{cov})^{<\omega}\to D_{cov}$ is a \textit{winning strategy for player II } in the  game $G_2$ whenever, for any sequence $U_0, U_1, \ldots\;$  consisting of non-empty open subsets with  
$  U_0\in\sigma_{cov}(\emptyset)=\Pee_0\in D_{cov}$ and   $U_n\in\sigma_{cov}(U_0, U_1, \ldots, U_{n-1})=\Pee_n\in D_{cov} $, for all $n\in \omega$,  there holds  $\cl\bigcup_{n\in\omega} U_n\not=X$. 

 In the paper \cite{dkz} the authors introduced an open-open game. We say that $G$ is \textit{an open-open game} of length $\omega$ if two players take turns playing; a round consists of player I choosing a non-empty open set $U\subseteq X$ and player II choosing a non-empty open $V\subseteq U$; I wins if the union of II's open sets is dense in $X$, otherwise II wins. Suppose that  there exists    a  function 
$$s_{op} :\bigcup \{ \Tee^n: n\geq 0\} \to \Tee $$  such that
 for each sequence  $V_0, V_1, \ldots\;$ consisting of non-empty elements of $\Tee$ with  
$  s_{op}(V_0)\subseteq V_0$ and   $s_{op}(V_0, V_1, \ldots, V_n)\subseteq V_{n} $, for all $n\in \omega$,  there holds  $\cl\bigcup_{n\in\omega} V_n\not=X$. Then the function $s_{op}$ is called a  \textit{winning strategy for II player} in the open-open game and we say that the space $X$ is \textit{II-favorable}.  

It is known \cite{dkz} that the open-open game $G$ is equivalent  to $G_2$. 
We consider only  games with the length equal to $\omega$. 
 In \cite{dkz} the authors introduced a game $G_7$ which is played as follows:  In the $n$-th inning II chooses
$\Oee_n$, a family of open sets with $\bigcup \Oee_n$  dense in $X$. I responds with $\mathcal{T}_n$,
a finite subfamily of $\Oee_n$; I wins  if
$\bigcup_{n\in\omega} \mathcal{T}_n$ is dense subset of $X$; otherwise, II wins. 

 According to A. Szyma\'{n}ski \cite{szy}  a sequence $\{\Pee_n:n\in\omega\}$  
of open families in $X$  is a \textit{tiny sequence} if 

\begin{enumerate}
	\item $\bigcup\Pee_n$ is dense in $X$ for all $n\in\omega$
	\item if $\mathcal F_n$ is a finite subfamily of $\Pee_n$ for each $n\in\omega$ then  $\bigcup\{\bigcup\mathcal F_{n}:n\in\omega\}$ is not dense in $X$.
\end{enumerate}

We call  a sequence $\{\Pee_n:n\in\omega\}$ of open families in $X$  a \textit{ 1-tiny sequence} if 

\begin{enumerate}
	\item $\bigcup\Pee_n$ is dense in $X$ for all $n\in\omega$
	\item if $F_n$ is a member of $\Pee_n$ for each $n\in\omega$ then  $\bigcup\{ F_{n}:n\in\omega\}$ is not dense in $X$.
\end{enumerate}

  M. Scheepers  used in the paper \cite{sch} negation of the existence of tiny sequence, and 1-tiny sequence - called these properties $S_{fin}(\mathcal D,\mathcal D)$ and $S_{1}(\mathcal D,\mathcal D)$ respectively. In this paper we refer to  notions tiny sequence and 1-tiny sequence, because in some situations
(Theorem  \ref{g7iffts} and \ref{g2iff1ts}) we can define them.
 
  Recall another game $G_4$ introduced in \cite{dkz}.  In the $n$-th inning player I chooses finite open family $\Aaa_n$. Player II responds with a finite, open family $\Bee_n$ with $|\Bee_n|=|\Aaa_n|$ and for each $V\in\Aaa_n$ there exists $W\in\Bee_n$ such that $W\subseteq V$. I wins if
$\bigcup_{n\in\omega} \bigcup\Bee_n$ is dense subset of $X$; otherwise, II wins. One can prove that the game $G_7$ is equivalent to the game $G_4$ in a way similar to the proof of the equivalence between games $G$ and $G_2$.

From now on we consider only c.c.c. spaces.

\begin{theorem}[M. Scheepers, Theorem 2 \cite{sch} ]\label{g7iffts}
II has a winning strategy in the game $G_7$ if and only if there exists a tiny sequence.
\end{theorem}

\begin{theorem}[M. Scheepers, Theorem 14 \cite{sch} ]\label{g2iff1ts}
Player II has a winning strategy in the game $G_2$ if and only if there exists a 1-tiny sequence.
\end{theorem}

\section{ The main results}

Recall that  $X$ is called a \textit{II-favorable} space if  player II a has winning strategy in the game G. If 
player I has a winning strategy in the game G then we say that the space is \textit{I-favorable}.

The following theorem was proven  by K. Kuratowski and S. Ulam, see \cite{ku}. 
In order to formulate it, let us recall that: a $\pi$-base is a family of open, nonempty sets such that any open set contains a set from this family, and the
$\pi$-weight of a space is the smallest cardinality of a $\pi$-base in this space. 
 \\ \indent 
\textit{Let $X$ and $Y$ be topological spaces such that $Y$ has countable $\pi$-weight. If $E\subseteq X\times Y$ is a nowhere dense set, then there is  a meager set $P\subseteq X$ such that the section $E_x= \{ y: (x,y) \in E \}$
 is nowhere dense in $Y$ for each point $x\in X\setminus P$}.

A space $Y$ is \textit{universally Kuratowski-Ulam} (for short, \textit{uK-U space}), whenever for a topological space $X$ and a nowhere dense  set $E\subseteq X\times Y$ the set $$ \{x\in X: \{y\in Y: (x,y)\in E\} \mbox{ is not nowhere dense  in } Y\}$$ is meager in $X$, 
   see D. Fremlin \cite{fem} (compare \cite{fnr}).  
   In the paper \cite{kp7} authors have shown that a compact I-favorable  space is universally Kuratowski-Ulam and  posed a question:
    \textit{Does there exist  a compact universally Kuratowski-Ulam space which is not 
 I-favorable}?
 We will partially answer to this question, namely we will prove that a II-favorable space is not universally Kuratowski-Ulam space.
 
 \begin{theorem}
Let $X$ be  a dense in itself space with a $\pi$-base $\Bee=\bigcup_{n\in\omega}\Bee_n$, where $\Bee_n$ is a maximal family of pairwise disjoint open sets for $n\in\omega$ and let  $Y$ be II-favorable c.c.c. space.  Then the 
Kuratowski-Ulam theorem does not hold in $X\times Y$.  
 \end{theorem}
\begin{proof}
By Theorem \ref{g2iff1ts} there is a 1-tiny sequence $\{\Pee_n:n\in\omega\}$. Since the  space $Y$ satisfies c.c.c. we can assume that each $\Pee_{n+1}$ is a countable, open, pairwise disjoint family. We can also assume that
every $\Pee_{n+1}$ is a refinement of $\Pee_n$, i.e. each member of $\Pee_{n+1}$ is a subset of a member of $\Pee_n$. Let $\{V^n_\sigma:\sigma\in{}^n\nat\}$ be such an enumeration of the family $\Pee_n$, that for each $\tau\in{}^{n-1}\nat$, $\{V^n_{\tau\frown k}\colon k\in\nat\}=\Pee_n$.

We can assume that $\Bee_{n+1}$ is a refinement of $\Bee_n$ and  $|\{V\in \Bee_{n+1}:V\subseteq U\}|\geq\omega$ for each 
$U\in \Bee_{n}$. For each $n\in\nat$ fix a function $f_n:\Bee_n\to {}^n\nat$   such that for a fixed $U\in\Bee_n$ we have 

 $$(1)\; \{f_{n+1}(V):V\in \Bee_{n+1}\mbox{ and }V\subseteq U\}=f_n(U)^\frown\nat.$$

\noindent
Therefore, there holds the condition:

	$$(2)\;\text{ if }V\subset U\text{ then }f_{n+1}(V)\supset f_n(U)\text{ for every }V\in\Bee_{n+1} \text{ and }U\in\Bee_n.$$

 Consider an open set
 $$F=\bigcup\{\bigcup\{U\times V^{n}_{f_n(U)}: U\in\Bee_n\}: n\in\nat\}.$$
 
We shall show that $F$ is dense and $F_x=\{y\in Y: (x,y)\in F\} $ is not dense for each $x\in X$. If $x\in X\setminus 
\bigcap\{\bigcup\Bee_n:n\in\nat\}$ then it is easy to see that $E_x$ is not dense. 
If $x\in\bigcap\{\bigcup\Bee_n:n\in\nat\}$ then by condition $(2)$ there is $\sigma\in{}^\nat\nat$ such that for each $n\in\nat$ there exists $U_n\in\Bee_n$ with 
$f_n(U_n)=\sigma|n$ and $x\in\bigcap\{ U_{n}:n\in\nat\}$, hence $F_x=\bigcup\{ V^n_{\sigma|n}:n\in\nat\}$. Since $V^n_{\sigma|n}\in\Pee_n$ for each $n\in\nat$ and  $\{\Pee_n:n\in\omega\}$ is a 1-tiny sequence the set $\bigcup\{ V^n_{\sigma|n}:n\in\nat\}$ is not dense.

Now we show that $F$ is a dense set. Let $U\times W$ be any open set. Since $\Bee$ is a $\pi$-base there are $n\in\nat$ and  $U_0\in\Bee_n$ such that $U_0\subseteq U$. Let $\sigma=f_n(U_0)$, since $\{V^{n+1}_{\sigma\frown k }\colon k\in\nat\}$ is a dense family, we get that $W\cap V^{n+1}_{\sigma\frown k }\neq \emptyset$ for some $k\in\nat$. By $(1)$, we may take $U_1\subseteq U_0$ such that $U_1\in\Bee_{n+1}$
and $f_{n+1}(U_1)=\sigma^\frown k$. Thus $U_1\times V^{n+1}_{f_{n+1}(U_1)}\cap U\times W\neq \emptyset$.
\end{proof}

Since $\mathbb R$ with natural topology satisfies assumption of the above theorem and every 
universally Kuratowski-Ulam space is c.c.c. space we get the following theorem.

 \begin{theorem}
A II-favorable space is not universally Kuratowski-Ulam space. 
 \end{theorem}

Following  \cite[p.86 - 91]{ox} recall category measure space. If $X$ is a topological space with finite measure $\mu$ defined on the $\sigma$-algebra $S$ of sets having the Baire property, and if $\mu(E)=0$
if and only if $E$ is of  a meager set, then $(X, S,\mu)$ is called a \textit{category measure space}. An  example of a regular Baire space which is a category measure space, is an open interval $(0,1)$ with Lebesgue measure $\mu_l$ and density topology $\Tee_d$, see \cite{ox}.  For density topology and measurable set $A\subseteq (0,1)$ the following conditions  are equivalent:
\begin{enumerate}
	\item $\mu_l(A)=0$,
	\item $A$ is closed and nowhere dense .
\end{enumerate}
In the space $((0,1),\Tee_d)$ there is a 1-tiny sequence but there is no tiny sequence. Indeed, define
a 1-tiny sequence in the following way: let $\Pee_n=\{U: U\in\Tee_d\mbox{ and } \mu_l(U)\leq \frac 1 {3^n}\}$. If 
$\{U_n:n\in\nat\}$ is a family chosen by player I then $\mu_l(\bigcup\{U_n:n\in\nat\})\leq \frac 1 2$. 
Therefore $\{U_n:n\in\nat\}$ is not a dense family. Now assume that there exists a tiny sequence  
$\{\Pee_n:n\in\nat\}$. In each stage we choose a finite subfamily $\Raa_n\subset \Pee_n$ such that
$ \mu_l(\bigcup\{\bigcup\Raa_i:i\leq n\})\geq 1-\frac 1 n $, hence we get a dense family $\bigcup\{\Raa_n:n\in\nat\}$.

The authors of the paper \cite{dkz} posed a question (Question 4.3): \textit{Does a  player have a winning strategy in the game 
$G$ if and only if the same player has a winning strategy in the game $G_7$.} The author of paper
\cite{sch} showed that if $\cov(\mathcal M)<\mathfrak{d}$ the answer is NO.  We show that  games $G$ and $G_7$ are not equivalent.

\begin{corollary}
The game $G$ is not equivalent to the game $G_7$
\end{corollary}

\begin{proof}
By Theorem \ref{g2iff1ts} a winning strategy of 
II player in the game $G$ is equivalent to the existence of a $1$-tiny sequence  
and by Theorem \ref{g7iffts} the existence of a winning strategy of  player II in the game $G_7$ is equivalent to the existence of a tiny sequence. Since in the space $((0,1),\Tee_d)$ there is a 1-tiny sequence but there is no tiny sequences we get that games $G$ and $G_7$ are not equivalent. 
\end{proof}
Since the game $G_7$ is equivalent to the game $G_4$,  we get the following:

\begin{corollary}
The game $G$ is not equivalent to the game $G_4$
\end{corollary}

\section{Some remarks}
It is known that on the  $\omega_1$ with discrete topology II player has a winning strategy in the game $G_7$,
but one can  pose a question:

\textit{Is it possible to construct a tiny sequence $\{\Pee_n\colon n\in\omega\}$ on a discrete space of the size $\omega_1$
with $|\Pee_n|=\omega $ for all $n\in\omega$ ?}

The following Remark \ref{de} gives us the answer - it is possible if and only if the dominating number is equal
$\omega_1$. This is reformulation of well know results about critical cardinal number,   see 
W. Just, A. W. Miller, M. Scheepers and P. J. Szeptycki \cite{jmss}; D. Fremlin, A. W. Miller \cite{fm}
and B. Tsaban \cite{tsa}. 

Recall that $f\leq^*g$ denotes that for almost all $n\in\omega$ holds $f(n)\leq g(n)$, where $f,g$ are functions defined on natural numbers.
A family $\Raa\subseteq {}^\omega\omega$ is a \textit{dominating} family if for each
$f\in {}^\omega\omega$ there is $g\in \Raa$ such that $f\leq^* g$. The \textit{dominating number} $\mathfrak d$ is the smallest cardinality of a dominating family: $$\mathfrak d =\min\{|\mathcal R|:\mathcal R \mbox{ is dominating }\}.$$

\begin{remark}\label{de}
The smallest cardinality $\kappa$ such that there exists a tiny sequence $\{\Pee_n\colon n\in\omega\}$ on the discrete space of the size $\kappa$ with $|\Pee_n|=\omega $ for all $n\in\omega$ is equal to $\mathfrak{d}$.
\end{remark}

\begin{proof}
Let $X$ be any discrete space for which there exists  a tiny sequence $\{\Pee_n\colon n\in\omega\}$. We can assume that every $\Pee_n$ is a partition of $X$
into countably many blocks $\{X^n_0, X^n_1,\ldots\}$, so we may define for each $x\in X$ a function $f_x\colon \omega\rightarrow \omega$
in the following way: $f_x(n)=k$ whenever $x\in X^n_k$. Take an arbitrary function $f\colon \omega\rightarrow \omega$, and  any $x\in X\setminus \bigcup\{\bigcup\{X^n_{k}\colon k\leq f(n)\}\colon n<\omega\}$, then $f$ is dominated by the function $f_x$. It shows that $\{f_x\colon x\in X\}$ is a dominating family, hence $|X|\geq \mathfrak{d}$.

Now, let $\mathcal{F}\subset \omega^\omega$ be a dominating family of the cardinality $\mathfrak{d}$. Without loss of generality assume that for each function $f\colon\omega\rightarrow \omega$ there is $g\in\mathcal{F}$ such that $f(n)<g(n)$ for all $n<\omega$. We treat $\mathcal{F}$ as a discrete 
topological space. For $n,k\in\omega$ put $A^n_k=\{f\in\mathcal{F}\colon f(n)\leq k\}$ and set $\mathcal{P}_n=\{A^n_k\colon k<\omega\}$.
Of course, each family $\mathcal{P}_n$ is increasing and has the union equal to $\mathcal{F}$. From each $\mathcal{P}_n$ take some single $A^n_{f(n)}$ where $f\colon\omega\rightarrow \omega$. 
If $\bigcup\{A^n_{f(n)}\colon n<\omega\}$ was equal to $\Fee$, then it would contain  such a function $g$ that $g(n)>f(n)$ for  all $n\in\omega$, but it is not the case. Therefore $\{\Pee_n\colon n\in\omega\}$ is a tiny sequence.
\end{proof}

Recall a definition of a Baire number $\cov(\mathcal M)$ for the ideal $\mathcal M$ of meager subsets of real line $\mathbb R$:
$$\cov(\mathcal M)=\min\{|\mathcal A|:\mathcal A\subseteq\mathcal M\mbox { and } \bigcup\mathcal A=\mathbb R\}.$$
T. Bartoszy\'{n}ski \cite{bar} proved that $\cov(\mathcal M)$ is the cardinality of the smallest
family $\mathcal F\subseteq {}^\omega\omega$ such that
$$ \forall(g\in {}^\omega\omega)\exists(f\in\mathcal F)\forall(n\in\omega) f(n)\not =g(n).$$
We get another well known characterization of such families by a 1-tiny sequence.

\begin{remark}\label{cov}
The smallest cardinality $\kappa$ such that there exists a 1-tiny sequence $\{\Pee_n\colon n\in\omega\}$ on the discrete space of the size $\kappa$
with $|\Pee_n|=\omega $ for all $n\in\omega$ is equal to $\cov(\mathcal M).$
\end{remark}
 We give the proof for the sake of completeness. We shall prove that the smallest cardinality of a family $\mathcal F\subseteq {}^\omega\omega$ such that
$$(*)\; \forall(g\in {}^\omega\omega)\exists(f\in\mathcal F)\forall(n\in\omega) f(n)\not =g(n)$$
is equal to the smallest cardinality $\kappa$ such that there exists a 1-tiny sequence $\{\Pee_n\colon n\in\omega\}$ on the discrete space $\kappa$
with $|\Pee_n|=\omega $ for all $n\in\omega$.

\begin{proof}
Let $\mathcal F=\{f_\alpha:\alpha<\kappa\}\subseteq {}^\omega\omega$ be a family with  the property $(*)$. 
Define $A^i_n=\{f\in\mathcal F: f(i)=n\}$ for every $i,n\in\omega$. Let 
$\Pee_i=\{A^i_n:n\in\omega\}$ for $i\in\omega$. We will show that $\{\Pee_i\colon i\in\omega\}$ is a 1-tiny sequence.
Assume that we have chosen  $A^i_{n_i}\in\Pee_i$  for each $i\in\omega$. Define a function 
$g(i)=n_i$ for $i\in\omega$.  Since $\mathcal F$ satisfies $(*)$  there is $f\in
\mathcal F$ such that $f(i)\not =g(i)$ for each $i\in\omega$. Therefore we get $f\in\mathcal F\setminus \bigcup\{A^i_{n_i}:i\in\omega\}$.
 
 Let $\{\Pee_n\colon n\in\omega\}$ be a 1-tiny sequence with  $|\Pee_n|=\omega $ and $\bigcup\Pee_n=\kappa$ for each $n\in\omega$. We can assume that each $\Pee_n$ consists of pairwise disjoint subsets of $\kappa$.
 Let $\{A^n_k: k\in\omega\}$ be a enumeration of $\Pee_n$.  We define a
 function $f_x\in {}^\omega\omega$ for each $x\in\kappa$ in the following way: 
 $f_x(i)=n ,$ where $ x\in A^i_n$ for each $i\in\omega.$
The family $\{f_x: x\in \kappa\}$ satisfies $(*)$. Indeed, let $g\in{}^\omega\omega$ be any function. Since $\{\Pee_n: n\in\omega\}$ is a 1-tiny sequence, choose $x\in\kappa\setminus \bigcup\{A^i_{g(i)}:i\in\omega\}$. Finally, observe that $f_x(i)\not =g(i)$ for every $i\in \omega$.
\end{proof}

We shall recall definition of the bounding number
$$ \mathfrak{b}=\min\left\{|\mathcal F|:\mathcal F\subseteq{}^\omega\omega\mbox{ and }\forall(g\in{}^\omega\omega)\exists(f\in\mathcal F)\neg(f\leq^* g))\right\}$$

We say that a sequence  $\{\Pee_n\colon n\in\omega\}$ of open families in $X$  is a $\mathfrak{b}$-\textit{tiny sequence} if 

\begin{enumerate}
	\item $\bigcup\Pee_n$ is dense in $X$ for all $n\in\omega$;
	\item if $\mathcal F_n$ is a finite subfamily of $\Pee_n$ for each $n\in\omega$, then there exists strictly increasing sequence  
	$\{n_i:i\in\omega\}$ such that 
$$\bigcup\left\{\bigcup\mathcal F_{n_i}:i\in\omega\right\}$$
 is not dense in $X$.
\end{enumerate}

We get the next reformulation of the bounding number.

 \begin{remark}\label{bounding}
The smallest cardinality $\kappa$ such that there exists a \\$\mathfrak{b}$-tiny sequence $\{\Pee_n: n\in\omega\}$ on the discrete space of the size $\kappa$
with $|\Pee_n|=\omega $ for all $n\in\omega$ is equal to $\mathfrak{b}$.
\end{remark}

{\bf Acknowledgment.} The authors are  indebted to the Referee for very careful reading of the paper and  valuable comments.

\end{document}